\documentclass{sig-alternate-05-2015}

\usepackage{etoolbox}
\makeatletter
\patchcmd{\maketitle}{\@copyrightspace}{}{}{}
\makeatother

\usepackage[sort]{natbib} % required for bibliography in the ifacconf class
\usepackage{enumitem} 
\usepackage{algorithm}
\usepackage{algorithmic}

\usepackage{graphicx}
\DeclareGraphicsExtensions{.pdf,.png,.jpg,.eps,.ps}
	
\usepackage{epstopdf}
\usepackage{tikz}
\usetikzlibrary{arrows,automata, shapes, snakes, positioning, decorations}
\usepackage{tkz-graph}
\def \IMG {images}
\usepackage{subcaption}

\usepackage{amsmath}						
\usepackage{amssymb}

\newcommand{\graph}{\mathbf{G}}
\newcommand{\reels}{\mathbb{R}}

\newcommand{\card}[1]{|{#1}|} % the number of - used alot, but we are never happy with the symbols - I made it tunable.
\newtheorem{theorem}{Theorem}%[section]
\newtheorem{definition}{Definition}%[section]
\newtheorem{lemma}{Lemma}
\newtheorem{remark}{Remark}
\newtheorem{proposition}{Proposition}
\newtheorem{corollary}{Corollary}

\newtheorem{example}{Example}

\providecommand{\com[2]}{\begin{tt}[#1: #2]\end{tt}}

\def\sharedaffiliation{%
\end{tabular}
\begin{tabular}{c}}

\begin{document}

%\setcopyright{acmcopyright}
%\conferenceinfo{HSCC '17}{April 18--21, 2017, Pittsburgh, PA, USA}

\title{Path-Complete Graphs and  Common Lyapunov Functions\titlenote{This work was supported by the French Community of Belgium and by the IAP network DYSCO. D.A. is .... N.A. M.P. if a F.N.R.S./F.R.I.A. Fellow.  R.J.  is a Fulbright Fellow and a FNRS Fellow.}
}
% For the authors: This is a double blind review!
\numberofauthors{4} %  in this sample file, there are a *total*

%% We need to comment out this block for double - blind submission.
\author{
\alignauthor $\,$
\alignauthor David Angeli
\alignauthor $\,$
\sharedaffiliation
	   \affaddr{Dept. of Electrical and Electronic Engineering, Imperial College London, UK}\\
       \affaddr{Dept. of Information Engineering, University of Florence, Italy. }\\
        \email{d.angeli@imperial.ac.uk}
\and  
\alignauthor Matthew Philippe
\alignauthor Nikolaos Athanasopoulos
\alignauthor Rapha\"{e}l M. Jungers \titlenote{Currently visiting UCLA, Department of Electrical Engineering, Los Angeles, USA.}
\sharedaffiliation
      \affaddr{ICTEAM - Dept. of Mathematical Engineering, Universit\'e catholique de Louvain, Louvain-la-Neuve, Belgium}  \\
      \email{ $\{$matthew.philippe, nikolaos.athanasopoulos,raphael.jungers$\}$@uclouvain.be}
}
% There's nothing stopping you putting the seventh, eighth, etc.
% author on the opening page (as the 'third row') but we ask,
% for aesthetic reasons that you place these 'additional authors'
% in the \additional authors block, viz.
\date{13 October 2016}

\maketitle

\begin{abstract}
A Path-Complete Lyapunov Function is an algebraic criterion 
composed of a finite number of functions, called its pieces, and a directed, labeled graph defining \emph{Lyapunov inequalities}  between these pieces. It provides a stability certificate for discrete-time switching systems under arbitrary switching.\\
In this paper, we prove that the satisfiability of such a criterion implies the existence of a Common Lyapunov Function, expressed as the composition of minima and maxima of the pieces of the Path-Complete Lyapunov function. The converse, however, is not true even for discrete-time linear systems: we present such a system where a max-of-2 quadratics Lyapunov function exists while no corresponding Path-Complete Lyapunov function with 2 quadratic pieces exists.\\
In light of this, we investigate when it is possible to decide if a Path-Complete Lyapunov function is less conservative than another.
By analyzing the combinatorial and algebraic structure of the graph and the pieces respectively, we provide  simple tools to decide when the existence of such a Lyapunov function implies that of another.
 \end{abstract}

\section{Introduction}
Switching systems are dynamical systems for which the state dynamics varies between different operating
modes. They find application in several applications and theoretical fields, see e.g. \cite{PhEsSODT,AhJuJSRA,LiMoBPIS, JuTJSR}. They take the form
\begin{equation}
x(t+1) = f_{\sigma(t)}(x(t))
\label{eq:swsys}
\end{equation}
where the state
$x(t)$ evolves in $\reels^n$. The \emph{mode} $\sigma(t)$ of the system at time $t$ takes value from a set  $\{1,\ldots,M\}$ for some integer $M$.  Each $i$ mode of the $M$ modes of the system is described by a continuous map $f_{i}(x) : \reels^n \rightarrow \reels^n$. We assume that $f_i(x)=0 \Leftrightarrow x = 0$ for all modes.

 In this paper, we study criteria guaranteeing that the system (\ref{eq:swsys}) is stable under \emph{arbitrary switching}, i.e. where the function $\sigma(\cdot)$, called the switching sequence, takes values in $\{1, \ldots, M\}$ at any time $t$. This analysis can be extended through the more general setting of \cite{PhEsSODT} (see \cite{KoTBWF}, \cite[Section 3.5]{PhEsSODT}).
We study the following notions of stability, where $x(t, \sigma(\cdot), x_0)$ is the state of the system \eqref{eq:swsys} at time $t$ with a switching sequence $\sigma(\cdot)$ and an initial condition $x_0 \in \reels^n$.
\begin{definition}
The system \eqref{eq:swsys} is Globally Uniformly  Stable if there is a $\mathcal{K}_{\infty}$-function\footnote{A function $\alpha(z)$ is of class $\mathcal{K}$ if it is continuous, strictly increasing, with $\alpha(0) = 0$. It is of class $\mathcal{K}_\infty$ if it is unbounded as well.} $\alpha : \reels^+ \mapsto \reels^+$ such that for all $x_0 \in \reels^n$, for all switching sequences $\sigma(\cdot)$ and for all $t \geq 0$,
$$\| x(t, \sigma(\cdot), x_0) \| \leq \alpha(\|x_0\|).$$

The system is Globally Uniformly Asymptotically  Stable if there is a $\mathcal{KL}$-function\footnote{A function $\beta(z,t)$ is of class $\mathcal{KL}$ if, for each fixed $t$, $\beta(z,t)$ is a $\mathcal{K}$-function in $z$, and for each fixed $z$, $\beta(z,t)$ is a continuous function of $t$, strictly decreasing with $\lim_{t \rightarrow \infty} \beta(z,t) = 0$.} $\beta : \reels^+ \times \reels^+  \mapsto \reels^+$ such that for all $x_0 \in \reels^n$, for all switching sequences $\sigma(\cdot)$ and for all $t \geq 0$,
$$\| x(t, \sigma(\cdot), x_0) \| \leq \beta(\|x_0\|,t).$$
\label{def:stability}
\end{definition}
The stability analysis of switching systems is a central and challenging question in control (see \cite{LiAnSASO} for a description of several approaches on the topic). The question of whether or not a system is uniformly globally stable is in general undecidable, even when the dynamics is switching \emph{linear} (see e.g.  \cite{JuTJSR,BlTsTBOA}).

A way to assess stability for switching systems is to use Lyapunov methods, with the drawback that they often provide conservative stability certificates.
 For example, for \emph{linear} discrete-time switching systems of the form
 $$x(t+1) = A_{\sigma(t)} x(t)$$
 it is easy to check for the existence of a \emph{common quadratic} Lyapunov function (see e.g. \cite[Section II-A]{LiAnSASO}). However, such a Lyapunov function may not exist, even though the system is asymptotically stable (see e.g. \cite{LiMoBPIS, LiAnSASO}). Less conservative parameterizations of candidate Lyapunov functions have been proposed, at the cost of greater computational effort (e.g. for linear switching systems, \cite{PaJaAOTJ} uses sum-of-squares polynomials, \cite{GoHuDMII} uses max-of-quadratics Lyapunov functions,  and \cite{AnLaASCF} uses polytopic Lyapunov functions). \emph{Multiple Lyapunov functions} (see \cite{BrMLFA,ShWiSCFS,JoRaCOPQ}) arise as an alternative to common Lyapunov functions.
In the case of linear systems, the multiple \emph{quadratic} Lyapunov  functions such as those introduced in \cite{BlFeSAOD,DaRiSAAC, LeDuUSOD, EsLeCOLS} hold special interest as checking for their existence boils down to solving a set of LMIs. The general framework of \emph{Path-Complete} Lyapunov functions was recently introduced in \cite{AhJuJSRA} in this context, for analyzing and unifying the approaches cited above.\\
A Path-Complete Lyapunov function is a multiple Lyapunov function composed of a finite set of \emph{pieces} $\mathcal{V} = (V_i)_{i = 1, \ldots, N}$, with $V_i : \reels^n \mapsto \reels^+$, and a set of \emph{valid Lyapunov inequalities} between these pieces. We assume there exist two $\mathcal{K}_\infty$-functions $\alpha_1$ and $\alpha_2$ such that
\begin{equation}
 \forall x \in \reels^n,\, \forall i \in \{1, \ldots, N\}, \, \alpha_1(\|x\|)\leq V_i(x)\leq \alpha_2(\|x\|).
\label{eq:VIsKappa}
\end{equation}
These Lyapunov inequalities are represented by a directed and labeled graph $\graph = (S,E)$, where $S$ is the set of nodes, and $E$ the set of edges of the graph.There is one node in the graph for each one of the pieces $(V_i)_{i \in \{1 , \ldots, N \}}$ of the Lyapunov function. An edge takes the form $(p,q,w) \in E$, where $p,q \in S$ are respectively  its source and destination nodes, and where $w$ is the \emph{label} of the edge. Such a label is a finite sequence of modes of the system (\ref{eq:swsys}) of the form $w = \sigma_1, \ldots, \sigma_k$, with $\sigma_i \in \{1, \ldots, M\}$, $1 \leq i \leq k$.\\
An edge as described above encodes the Lyapunov inequality\footnote{We consider here certificates for Global Uniform Stability.  Analogous criteria for Global Uniform Asymptotic Stability can be obtained  with strict inequalities in (\ref{eq:ValidLyapInequality}).}
\begin{equation}
(p,q,w) \in E \Rightarrow \, \forall x \in \reels^n, \, V_q(f_w(x)) \leq V_p(x),
\label{eq:ValidLyapInequality}
\end{equation}
where $1 \leq p,q \leq P$ and for $w = \sigma_1 \ldots \sigma_k$, with $\sigma_i \in \{1, \ldots, M\}$, and $f_w = f_{\sigma_k} \circ \cdots \circ f_{\sigma_1}$ (see Figure \ref{fig:edgeInAGraph}).
\begin{figure}[!ht]
\centering
\begin{tikzpicture}[->,>=stealth',shorten >=1pt,auto,node distance=2cm,
                    semithick]
 
  \node[state] (Vi)                           {$p$};
  \node[state] (Vj)     [right of = Vi]       {$q$};
  \path (Vi)edge [bend left]             node {$w = \sigma_1\ldots\sigma_k$} (Vj);
\end{tikzpicture}
\caption{The edge encodes  $V_q(f_w(x)) \leq V_p(x)$.}
\label{fig:edgeInAGraph}
\end{figure}
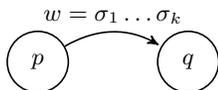
By transitivity, paths in the graph $\graph$ encode  Lyapunov inequalities as well. Given a path $p = (s_i, s_{i+1},w_i)_{i = 1 , \ldots, k}$ of length $k$, we define the \emph{label} of the path as the sequence $w_1\ldots w_k$ (i.e. the concatenation of the sequences  on the $k$ edges). Such a path
encodes the inequality $V_{s_{k+1}}(f_{w_k} \circ \cdots \circ f_{w_1}(x)) = V_{s_{k+1}}(f_{w_1\ldots w_K}(x)) \leq V_{s_1}(x)$.\\
The graph $\graph$ defining a Path-Complete Lyapunov function has a special structure, which is defined below and is illustrated in Figure \ref{fig:PC}.

\begin{definition}[Path-Complete Graph]
Consider a directed and labeled graph $\graph = (S,E)$,  with edges $(s,d,w) \in E$ with $s,d \in S$ and the label $w$ is a finite sequence over $\{1, \ldots, M\}$. The graph is \emph{path-complete} if for any finite sequence $w$ on $\{1, \ldots, M\}$, there is a path in the graph with a label $w'$ such that  $w$ is contained in $w'$.
\label{def:PC}
\end{definition}

 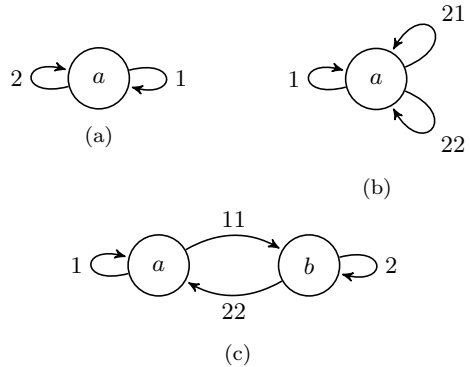
\begin{figure}[!ht]
\centering
\begin{subfigure}[c]{0.4\columnwidth}
\centering
\begin{tikzpicture}[->,>=stealth',shorten >=1pt,auto,node distance=2.8cm,
                   semithick, scale = 1, transform shape ]
 \node[state] (V1)                           {$a$};

 \path (V1)edge [loop right] node {$1$}             (V1)
     	   edge [loop left]  node {$2$}             (V1)
     	   ;
\end{tikzpicture}
\caption{}
\label{fig:pc*}
\end{subfigure}
~
\begin{subfigure}[c]{0.4\columnwidth}
\centering
\begin{tikzpicture}[->,>=stealth',shorten >=1pt,auto,node distance=2.8cm,
                   semithick, scale = 1, transform shape ]
 \node[state] (V1)                           {$a$};

 \path (V1)edge [loop left] node {$1$}             (V1)
     	   edge [in = 60, out = 22.5, loop]  node [above right] {$21$}             (V1)
 		   edge [in = 300, out = 337.5, loop]  node {$22$}             (V1)
     	   ;
\end{tikzpicture}
\caption{}
\label{fig:pc2}
\end{subfigure}
\centering
\begin{subfigure}[c]{\columnwidth}
\centering
\begin{tikzpicture}[->,>=stealth',shorten >=1pt,auto,node distance=2cm,
                    semithick, scale = 1, transform shape ]
 \node[state] (V1)                           {$a$};
 \node[state] (V2)     [right of = V1]       {$b$};
  \path (V1)edge [looseness = 10, loop left]             node {$1$} (V1)
      	    edge [bend left] node {$11$}            (V2)
     	(V2)edge [bend left]             node {$22$} (V1)
      	    edge [loop right] node {$2$} (V2);
\end{tikzpicture}
\caption{}
\label{fig:notpc}
\end{subfigure}
\caption{The graphs on Figure \ref{fig:pc*} and \ref{fig:pc2} are both path-complete, but the graph on Figure \ref{fig:notpc} is not: there are no paths containing the finite sequence $1212$. }
\label{fig:PC}
\end{figure}

It is shown in \cite[Theorem 2.4]{AhJuJSRA} that a  Path-Complete Lyapunov function is indeed a sufficient stability certificate for a switching system\footnote{While the cited result relates to linear systems and homogeneous Lyapunov functions, it extends directly to the more general setup studied here.}. Interestingly, it was recently shown in \cite{JuAhACOL} that, for linear systems, given a candidate multiple Lyapunov function with quadratic pieces $(V_i)_{i = 1, \ldots, N}$ and with Lyapunov inequalities encoded by a graph $\graph$, we cannot conclude stability \emph{unless} $\graph$ is path-complete.

In this paper we first ask a natural question which aims to reveal the connection to classic Lyapunov theory:
Can we extract a Common Lyapunov function for the system (\ref{eq:swsys}) from  a Path-Complete Lyapunov function?  We answer this question affirmatively in Section \ref{section:embedded}, and show that we can always extract a Lyapunov function which is of the form
\begin{equation}
 V(x)  = \min_{S_1, \ldots, S_k \subseteq S} \left ( \max_{s \in S_i} V_s(x) \right ).
 \label{eq:embededLyapunovFunction}
\end{equation}
Our proof is constructive and makes use of a classical tool from automata theory, namely the \emph{observer automaton}, to form subsets of nodes in $\graph$ that interact in a well defined manner.
Next, we show in Subsection \ref{subsec:converse} that the converse does not hold. In detail, we show that there is an asymptotically stable linear system that has a max-of-2-quadratics Lyapunov function,
but for which no Path-Complete  max-of-2-quadratics Lyapunov function exists.
In Section \ref{section:partial_order} we turn our attention to the problem of deciding a priori when a candidate Path-Complete Lyapunov function provides less conservative stability certificates than another.
By analyzing the combinatorial and algebraic structure of the graph and the pieces respectively, we provide tools in Subsections \ref{subsec:Bij} and \ref{subsec:Simu} to decide when the existence of such a Lyapunov function implies that of another.
We illustrate our results numerically in Section \ref{Sec:NumericalExample}, and draw the conclusions in Section~\ref{conclusions}.

\section{Preliminaries}
\label{section:preliminaries}

Given any integer $M \geq 1$, we write  $[M] = \{1, \ldots, M\}$.
For the sake of exposition, the directed graphs $\graph = (S,E) $ considered herein have the following property: the labels on their edges are of length 1, i.e., for $(i,j,w) \in E$, $w \in [M]$ (which is not the case, e.g. for the graph of Figure \ref{fig:pc2}).
It is easy to extend our results to the more general case, as shown in Remark \ref{remark:expanded} later. \\
We use several tools and concepts from \emph{Automata theory} (see e.g. \cite[Chapter 2]{CaLaITDE}).
\begin{definition}[Connected graph]
The graph $G = (S,E)$ is strongly connected if for all pairs $p,q \in S$, there is a path from $p$ to $q$.
\end{definition}
\begin{definition}[(Co)-Deterministic Graph]
$\,$\\A graph $G = (S,E)$ is \emph{deterministic} if for all $s \in S$, and all $\sigma \in [M]$, there is \emph{at most one} edge $(s, q, \sigma) \in E$.\\
The graph is \emph{co-deterministic} if for all $q \in S$, and all $\sigma \in [M]$, there is at most one edge $(s, q, \sigma) \in E$.
\label{def:(co)deterministic}
\end{definition}
\begin{definition}[(Co)-Complete Graph]
$\,$\\A graph  $G = (S,E)$ is \emph{complete} if for all $s \in S$, for all $\sigma \in [M]$ there exists \emph{at least one} edge $(s, q, \sigma) \in E$.\\
The  graph is \emph{co-complete} if for all $q \in S$, for all $\sigma \in [M]$, there exists at least one edge $(s, q, \sigma) \in E$.
\label{def:(co)complete}
\end{definition}
A (co)-complete graph is also path-complete  \cite[Proposition 3.3]{AhJuJSRA}.
The following allows us to dissociate the graph of a Path-Complete Lyapunov function from its pieces:
\begin{definition}
Given a system \eqref{eq:swsys}, a graph $\graph = (S,E)$ and a set of functions $\mathcal{V} = (V_s)_{s \in S}$, we say that $\mathcal{V}$ is a \emph{solution}
for $\graph$, or equivalently,  $\graph$ is feasible for $\mathcal{V}$, if for all $(p,q,\sigma) \in E$, $V_q(f_\sigma(x)) \leq V_p(x).$
\end{definition}
Whenever clear from the context, we will make all references to the system (\ref{eq:swsys}) implicit.

\section{Induced common Lyapunov functions}
\label{section:embedded}

As defined in the introduction, a Path-Complete Lyapunov function is a type of multiple Lyapunov function with a path-complete graph $\graph = (S,E)$ describing Lyapunov inequalities of the form (\ref{eq:ValidLyapInequality}) between its pieces $(V_s)_{s \in S}$.\\ In this section, we show that we can always extract from a Path-Complete Lyapunov function an \emph{induced common Lyapunov function} $V(x)$ for the system, that satisfies
$$\forall x \in \reels^n, \, \forall \sigma \in [M], \, V(f_\sigma(x)) \leq V(x).$$
To do so, we make use of the concept of \emph{observer automaton} \cite[Section 2.3.4]{CaLaITDE}, adapted for directed and labeled graphs (see Remark \ref{rem:GraphAndAutomata}).  This graph is  defined as follows, and its construction is illustrated in Example \ref{ex:observer}.

\begin{definition}[Observer Graph]
Consider a graph $\graph = (S,E)$. The \emph{observer} graph $O(\graph) = (S_O, E_O)$ is a graph where each state corresponds to a subset of $S$, i.e. $S_O \subseteq 2^S$, and is constructed as follows:
\begin{enumerate}
\item[1.] Set  $S_O := \{ S \}$ and $E_O := \emptyset$.
\item[2.] Set $X:=\emptyset$. For each pair $(P,\sigma) \in S_O \times [M]$: 
\begin{enumerate}
\item Compute $Q := \cup_{p \in P} \{q | \, (p,q,\sigma) \in E\}.$
\item If $Q \neq \emptyset$, set $E_O:=E_O\cup\{(P,Q,\sigma)\}$ then  $X:=X\cup Q$.
\end{enumerate}
\item[3.] If  $X \subseteq S_O$, then the observer is  given by $O(\graph) = (S_O, E_O)$. Else, set
 $S_O := S_O \cup X$ and go to step 2.
\end{enumerate}
\label{def:observers}
\end{definition}

We stress that the nodes of the observer graph $O(\graph)$ correspond to \emph{sets of nodes} of the graph $\graph$.

\begin{example}
Consider the graph $\graph$ of Figure \ref{fig:graphExample}.
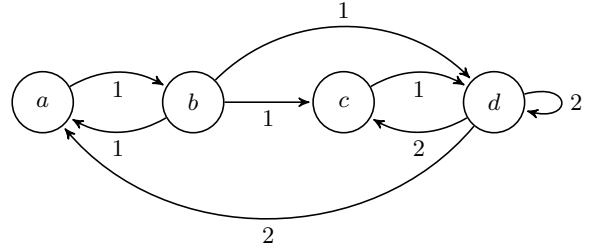
\begin{figure}[!ht]
\centering
\begin{tikzpicture}[->,>=stealth',shorten >=1pt,auto,node distance=2cm,
                    semithick, scale = 1, transform shape ]
 \node[state] (a)                          {$a$};
 \node[state] (b)     [right of = a]       {$b$};
 \node[state] (c)     [right of = b]       {$c$};
 \node[state] (d)     [right of = c]       {$d$};

  \path (a) edge [bend left]        node [below]{$1$} (b)
  	    (c) edge [bend left] node [below]{$1$} (d)
  	    (b) edge [bend left] node [below]{$1$} (a)
  	    	edge  node [below]{$1$} (c)
  	    	edge [bend left = 45]node [above] {$1$} (d)
  	    (d) edge [bend left = 50]node{$2$} (a)
  	    	edge [bend left] node{$2$} (c)
  	    	edge [loop right] node{$2$} (d)
  	    	;
\end{tikzpicture}
\caption{A path-complete graph on 4 nodes \emph{a,b,c,d} and 2 modes, for Example \ref{ex:observer}.}
\label{fig:graphExample}
\end{figure}
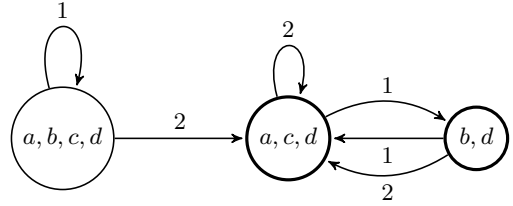
\begin{figure}[!ht]
\centering
\begin{tikzpicture}[->,>=stealth',shorten >=1pt,auto,node distance=2cm,
                    semithick, scale = 1, transform shape ]
 \node[state] (abcd)                          {$a,b,c,d$};
 \node[state, very thick] (acd)     [node distance = 3cm, right of = abcd]       {$a,c,d$};
 \node[state, very thick] (bd)     [node distance = 2.5cm, right of = acd]       {$b,d$};

  \path (abcd) edge [loop above] node {$1$} (abcd)
  				edge node {$2$} (acd)
  		(acd) edge [loop above] node {$2$} (acd)
  			  edge [bend left] node {$1$} (bd)
  	    (bd) edge node {$1$} (acd)
  	    	edge [bend left] node{$2$} (acd)
  	    	;
\end{tikzpicture}
\caption{Observer graph constructed from the graph $\graph$ on Figure \ref{fig:graphExample}. Each node of the observer $O(\graph)$ is associated to a set of nodes of $\graph$. Notice that the subgraph on the nodes $\{a,c,d\}$, $\{b,d\}$ is itself a \emph{complete} graph.}
\label{fig:graphExample_Obs}
\end{figure}
The observer graph $O(\graph)$ is given on Figure \ref{fig:graphExample_Obs}. The first run through step 2 in Definition \ref{def:observers} is as follows. We have $P = S$. For $\sigma = 1$ the set $Q$ is again $S$ itself: indeed, each node $s \in S$ has at least one inbound edge with the label $1$. For $\sigma = 2$, since node $b$ has no inbound edge labeled $2$, we get $Q = \{a,c,d\}$. This set is then added to $S_O$ in step 3, and the algorithm repeats step 2 with the updated $S_O$.\\
\label{ex:observer}
\end{example}

\begin{remark}
The notion of the \emph{observer automaton} is presented in \cite[Section 2.3.4]{CaLaITDE}. Generally, an automaton is represented by a directed labeled graph with a \emph{start} state and one or more \emph{accepting} states. The graphs considered here can be easily transformed into non-deterministic automata by using the so-called $\epsilon$-transitions (see \cite[Section 2.2.4]{CaLaITDE} for definitions). Given a graph $\graph=(S,E)$, one can add $\epsilon$-transitions from a new node ``\emph{a}'' to all node in S and from all nodes in S to a new  node ``\emph{b}''. The generated automaton has the node ``\emph{a}'' as the start state and the node ``\emph{b}'' as the (single) accepting state.
\label{rem:GraphAndAutomata}
\end{remark}

Observe that in Figure \ref{fig:graphExample_Obs} the subgraph of $O(\graph)$ with  two nodes $\{a,c,d\}$ and $\{b,d\}$ is complete and strongly connected. This is due to a key property of the observer graph.  We suspect that this property is known (maybe in the automata theory literature) but we have not been able to find a reference until now.

\begin{lemma}
The observer graph $O(\graph) = (S_O, E_O)$ of any path-complete graph $\graph = (S,E)$ contains a \emph{unique} sub-graph $O^\star(\graph) = (S_O^\star, E_O^\star)$ which is \emph{strongly connected, deterministic and complete}.
\label{lemma:observerSubgraph}
\end{lemma}
\begin{proof}
The fact that the observer automaton has a complete, deterministic, connected component is well known \cite[p.90]{CaLaITDE}. From Remark \ref{rem:GraphAndAutomata}, the result extends as well to the observer graph. \\
We prove that this component is unique. For the sake of contradiction, we assume that the observer graph has two complete and deterministic connected components  $\graph_1 = (S_{O,1}, E_{O,1})$ and $\graph_2= (S_{O,2}, E_{O,2})$. Each component is itself a path-complete graph. Moreover, since they are deterministic and complete,  there can never be a path from one component to another.

For any sequence $w$ of elements in $[M]$, there exists a \emph{unique} path in $O(\graph)$ with source $S \in S_O$ and label $w$. Since $\graph_1$, $\graph_2$ are in $O(\graph)$, then by construction, there  exist two sequences $w_1$ and $w_2$ such that there is a path from $S \in S_O$ with label $w_1$ that ends in a node in $\graph_1$ and a path with label $w_2$ that ends in a node in $\graph_2$.

We now consider two paths of infinite length which start from $S \in S_O$. The first has the label $w_1w_2w_1\ldots$, illustrated below,
$$S \rightarrow_{w_1} P^1  \rightarrow_{w_2} Q^1 \rightarrow_{w_1} P^2 \rightarrow_{w_2} Q^2 \cdots  $$
and visits the nodes $P^i \in S_{O,1}$ and $Q^i  \in S_{O,1}$  after the $ith$ occurrence of the sequences $w_1$ and $w_2$ respectively.
The second path has the label $w_2w_1w_2\ldots$, illustrated below,
$$S \rightarrow_{w_2} R^1  \rightarrow_{w_1} T^1 \rightarrow_{w_2} R^2 \rightarrow_{w_2} T^2 \cdots  $$
and visits  $R^i \in S_{O,2}$ and $T^i \in S_{O,2}$ after the $i$-th occurrence of the word $w_2$ and $w_1$ respectively.

Since $\graph_1$ and $\graph_2$ are disconnected, we know that $S \neq P^i \neq T^i $ and $ S \neq Q^i \neq R^i$.
 Thus, $P^1 \subset S$ which in turn implies $Q^1 \subset R^1$,  $P^2 \subset T^1$ and so on. More generally, for all $i$, it holds that $Q^i \subset R^i$ and $P^{i+1} \subset T^i$ for all $i$. By symmetry, we have that  $T^i \subset P^i$ and $R^{i+1} \subset Q^i$. Consequently, we observe that\footnote{We denote the cardinality of a discrete set $P$ by $|P|$.} $|P^{i+1}| \leq |P^{i}| - 2$, thus, necessarily, $|P^{|S|-1}|=0$, which is a contradiction since $O(\graph)$ by construction cannot have empty nodes. Thus,  $O(\graph)$ has a unique, strongly connected, deterministic and complete sub-graph.
\end{proof}

We are now in position to introduce our main result.

\begin{theorem}[Induced Common Lyapunov Function]
Consider Path-Complete Lyapunov function with graph $\graph = (S,E)$ and pieces $\mathcal{V} = (V_s)_{s \in S} $ for the system (\ref{eq:swsys}).
Let $O^\star(\graph) = (S^\star_O, E^\star_O)$ be the complete and connected sub-graph of the observer $O(\graph)$.
Then, the function
\begin{equation}
V(x) = \min_{Q \in S^\star_O} \left ( \max_{s \in Q} V_s(x) \right )
\label{eq:minmaxclf}
\end{equation}
is a \emph{Common Lyapunov function} for the system \eqref{eq:swsys}.
\label{thm:EmbeddedCLF}
\end{theorem}

The result is illustrated in the following example, and its proof is provided in Subsection \ref{subsection:HiddenInequalitiesAndProofs}.

\begin{example}
Consider the graph $\graph$ of Figure \ref{fig:graphExample} and its observer graph in Figure \ref{fig:graphExample_Obs}. For this observer graph, the unique, strongly connected, deterministic and complete component $O^\star(\graph) = (S_O^\star, E_O^\star)$ has $S_O^\star = \{ \{a,c,d\}, \, \{b,d\}\}$.
Thus, if $\graph$ is feasible for a set of functions $\mathcal{V} = \{V_a, V_b, V_c, V_d\}$, from Theorem \ref{thm:EmbeddedCLF}, we conclude that
\begin{equation}
V(x) = \min \left \{ \max \left ( V_a(x), V_c(x), V_d(x) \right ), \, \max \left ( V_b(x), V_d(x) \right ) \right \}
\label{eq:exampleLyapunovFunction}
\end{equation}
is a Common Lyapunov function.
Figure \ref{fig:levelSetExample} illustrates an example of the level sets of the function  \eqref{eq:exampleLyapunovFunction} when each piece is a quadratic function. Note that this level set is not convex, which shows the expressive power of path-complete criteria.
A geometric illustration of the Lyapunov inequalities infered by the graph $\graph$, and in particular of the fact that $V(f_1(x)) \leq V(x)$, is presented in Figure \ref{fig:propagationExample}.
\begin{figure}[!ht]
\centering
\begin{subfigure}[b]{\columnwidth}
\centering
\includegraphics[scale = 0.3]{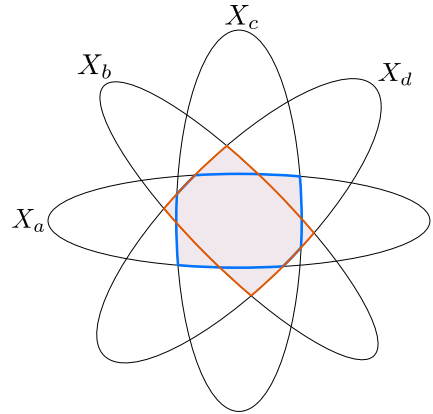}
\caption{A graphical illustration of the level set of $V$ at Eq. \eqref{eq:exampleLyapunovFunction} in Example Example \ref{ex:observer_cont}. The unit sublevel sets $\emph{X}_s$, $s \in \{a,b,c,d\}$ of the functions $(V_s)_{s \in S}$ are  ellipsoids. The level set of $V(x)$ is the union of two sets: the set $X_a \cap X_b \cap X_c$ (in blue) and the set $X_c \cap X_d$ (in orange).}
\label{fig:levelSetExample}
\end{subfigure}

\begin{subfigure}[b]{\columnwidth}
\centering
\includegraphics[scale = 0.3]{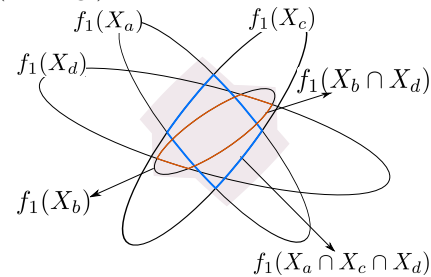}
\caption{
A graphical illustration of the Lyapunov inequalities for Example {\ref{ex:observer_cont}}. Let $(X_s)_{s \in S}$, the level sets of the functions $(V_s)_{s \in S}$ be such as in Figure \ref{fig:levelSetExample}. The image of the set $X_s$ set through $f_1$ is $f_1(X_s) = \{f_1(x), \, x \in X_s\}$. From an edge $(s,d,1) \in E$ of the graph $\graph$, we can infer that $f_1(X_s) \subseteq X_d$ since $V_s(x) \geq V_d(f_1(x))$. We can infer more refined relations by taking several edges. For example, from the edges $(b,a,1)$, $(b,c,1)$, and $(b,d,1)$, we infer that $f_1(X_b) \subseteq X_a \cap X_c \cap X_d$. Taking all edges of the form $(s,d,1)$ into account, we observe that  the level set of $V(x)$ (in gray) at Eq. (\ref{eq:embededLyapunovFunction}) is mapped into itself through $f_1$.
}
\label{fig:propagationExample}
\end{subfigure}
\caption{Illustrations for Example 2.}
\label{fig:Example2}
\end{figure}
\label{ex:observer_cont}
\end{example}

\subsection{Existence of an induced Common Lyapunov Function}
\label{subsection:HiddenInequalitiesAndProofs}

The following results expose relations between two subsets of states of a graph $\graph = (S,E)$ that lead to Lyapunov inequalities between the corresponding subsets of pieces of a Path-Complete Lyapunov function. These intermediate results are central to the proof of Theorem \ref{thm:EmbeddedCLF}.

\begin{proposition}
Consider the system \eqref{eq:swsys} and a graph $ \graph =(S,E)$ which is feasible for a set of functions $(V_s)_{s \in S}$. Take two subsets $P$ and $Q$ of $S$. If there is a label $\sigma$ such that
\begin{equation}
\forall p \in P, \,\exists q \in Q: \, (p,q,\sigma) \in E,
\label{eq:left-total_relation}
\end{equation}
 then
$$ \min_{q \in Q} V_q(f_\sigma(x)) \leq \min_{p \in P}{V_p(x)}.$$
\label{prop:hiddenEq:complete}
\end{proposition}
\begin{proof}
Take any $x\in\reels^n$. There exists a node $p^\star\in P$ such that $\min_{p\in P}V_p(x) = V_{p^\star}(x)$.
Also, there is at least one edge $(p^\star,q^\star,\sigma)\in E$, with $q^\star\in Q$. Thus, $V_{q^\star}(f_\sigma(x))\leq V_{p^\star}(x)$
and taking into account that  $\min_{q\in Q}V_{q}(f_\sigma(x))\leq V_{q^\star}(f_\sigma(x))$ the result follows.
\end{proof}
Proposition \ref{prop:hiddenEq:complete} generalizes  the following result, first stated in \cite[Corollary 3.4]{AhJuJSRA}.
\begin{corollary}
If $\graph = (S,E)$ is \emph{complete} and feasible for a set $(V_s)_{s \in S}$, then $\min_{s \in S} V_s(x)$ is a common Lyapunov function for the system (\ref{eq:swsys}).
\label{cor:CLFcomplete}
\end{corollary}
\begin{proof}
Proposition \ref{prop:hiddenEq:complete} holds here for $P = Q = S$, and all modes $\sigma \in [M]$.
\end{proof}
\begin{proposition}
Consider the system \eqref{eq:swsys} and a graph $\graph =(S,E)$ which is feasible for a set of functions $(V_s)_{s \in S}$. Take two sets of nodes $P$ and $Q$. If there is a label $\sigma$ such that,
\begin{equation}
\forall q \in Q, \exists p \in P: \, (p,q,\sigma) \in E,
\label{eq:right-total_relation}
\end{equation} then
$$ \max_{q \in Q} V_q(f_\sigma(x)) \leq \max_{p \in P}{V_p(x)}.$$
\label{prop:hiddenEq:cocomplete}
\end{proposition}
\begin{proof}
Take any $x\in\reels^n$. There exists a node $q^\star\in Q$ such that $\max_{q\in Q}V_q(f_\sigma(x))=V_{q^\star}(f_\sigma(x))$.
Also, since there exists a node $p^\star\in P$ such that $(p^\star,q^\star,\sigma)\in E$, it holds that $V_{q^\star}(f_\sigma(x))\leq V_{p^\star}(x)\leq \max_{p\in P}V_p(x)$ and the result follows.
\end{proof}
Proposition \ref{prop:hiddenEq:cocomplete} generalizes the following result, first stated in \cite[Corollary 3.5]{AhJuJSRA}.
\begin{corollary}
If $\graph = (S,E)$ is \emph{co-complete} and feasible for a set $(V_s)_{s \in S}$, then $\max_{s \in S} V_s(x)$ is a common Lyapunov function for the system.
\label{cor:CLFcocomplete}
\end{corollary}
\begin{proof}
Proposition \ref{prop:hiddenEq:cocomplete} holds here for $P = Q = S$, and all modes $\sigma \in [M]$.
\end{proof}
We are in the position to prove Theorem \ref{thm:EmbeddedCLF}.

\begin{proof}[of Theorem \ref{thm:EmbeddedCLF}]
Take a Path-Complete Lyapunov function with a  graph $\graph = (S,E)$ and pieces $(V_s)_{s \in S}$. Then, construct the \emph{observer graph} $O(\graph) = (S_O,E_O)$. By definition, there is an edge $(P,Q,\sigma) \in E_O$ if and only if $Q = \cup_{p \in P} \{q | \, (p,q,\sigma) \in E\},$ and therefore, the following property holds for such edges: $\forall q \in Q, \, \exists  p \in P$ such that  $(p,q,\sigma) \in E$.
Consequently, from Proposition \ref{prop:hiddenEq:cocomplete}, we have that
$$(P,Q,\sigma) \in E_O \Rightarrow  \max_{q \in Q} V_q(f_\sigma(x)) \leq \max_{p \in P}{V_p(x)}. $$
Therefore, the graph $ O(\graph) $ is \emph{feasible} for the set of functions $\mathcal{W} = \{W_P(x)\}_{P\in S_O}$, where
 $$W_P(x)=\max_{p\in P }V_p(x), \quad \forall P\in S_O.$$
From Lemma \ref{lemma:observerSubgraph}, there exists a sub-graph $O^\star(\graph) = (S^\star_O, E^\star_O)$ of $O(\graph)$ (with $S^\star_O \subseteq S_O$) which is complete and strongly connected.
Since $O(\graph)$ is feasible for  $\mathcal{W}$, its subgraph $O^\star(\graph)$ is feasible for the $\{W_P\}_{P \in S^\star_O}$.
Finally, since by Lemma \ref{lemma:observerSubgraph} $O^\star(\graph)$ is complete, we apply Corollary \ref{prop:hiddenEq:complete} and deduce that the function $W(x) = \min_{P \in S_O^\star} W_P(x)$ is a common Lyapunov function for the system.
\end{proof}

\begin{remark}
Our results extend to graphs $\graph = (S,E)$ where the labels are finite sequences of elements in $[M]$ (e.g., as in Figure \ref{fig:pc2}, 2c) as follows.
One can apply the results on the so-called \emph{expanded form} of these graphs  \cite[Definition 2.1]{AhJuJSRA}. The idea there is the following: if an edge $(p,q,w) \in E$ has a label $w = \sigma_1, \ldots, \sigma_k$ of length $k \geq 2$, then it is replaced by a path of length $k$, $ (s_i,s_{i+1},\sigma_i)_{i=1,\ldots,k}$ where $\sigma_1 = p$, $\sigma_{k+1} = q$, by adding the nodes $s_{2}, \ldots, s_k$ to the graph. The expanded form is obtained by repeating the process until all labels in the graph are of size 1 (see Figure \ref{fig:expanded_form}).\\
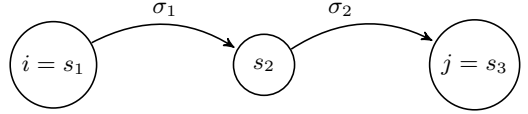
\begin{figure}[!ht]
\centering
\begin{tikzpicture}[->,>=stealth',shorten >=1pt,auto,node distance=2.8cm,
                    semithick]
 
  \node[state] (Vi)                           {$i = s_1$};
  \node[state] (V2)      [right of = Vi]      {$s_2$};  
  \node[state] (Vj)      [right of = V2]      {$j = s_{3}$};
  \path (Vi)edge [bend left]             node {$\sigma_1$} (V2);
    \path (V2)edge [bend left]             node {$\sigma_2$} (Vj);
\end{tikzpicture}
\caption{An edge in $\graph$ $(i,j,\sigma_1\sigma_2)$ (with label of length 2), is replaced by a path of length 2 in the extended form.}
\label{fig:expanded_form}
\end{figure}
If the graph $\graph = (S,E)$ is feasible for a set $\mathcal{V}$, we can always construct a set of functions $\mathcal{W}$ such that the expanded graph $\graph_e = (S_e, E_e)$ of $\graph$ is feasible for $\mathcal{W}$. For example, for a path   $(s_i,s_{i+1},\sigma_i)_{i = 1 , \ldots, k}$ in the expanded form corresponding to an edge $(p,q,w)$ in $\graph$ with $w = \sigma_1 \ldots \sigma_k$, we set $W_{p} = V_p$, $W_{q} = V_q$, and $W_{s_i}(x) = V_j(f_{\sigma_{i+1} \ldots \sigma_k}(x))$.
In Figure \ref{fig:expanded_form}, we would have $W_{s_2}(x) = V_j(f_{\sigma_2}(x))$.
\label{remark:expanded}
\end{remark}

\begin{remark}
We can establish a \emph{`dual'} version of the Theorem \ref{thm:EmbeddedCLF}. In specific, given the graph $\graph$, we reverse the direction of the edges obtaining a graph $\graph^\top$, construct  its observer $O(\graph^\top)$ and reverse the direction of its edges again, obtaining a graph $O(\graph^\top)^\top$.
 This graph is co-deterministic and contains a unique, strongly-connected, co-complete sub-graph that induces a Lyapunov function of the form $$
 V(x)  = \max_{S_1, \ldots, S_k \subseteq S} \left ( \min_{s \in S_i} V_s(x) \right ),
$$
which is, in general, \emph{not} equal to the common Lyapunov function obtained through Theorem \ref{thm:EmbeddedCLF}.
%The exposition and study of this dual result will be part of an extended version of this work.
\label{rem:observerReachability}
\end{remark}

\subsection{The converse does not hold}
\label{subsec:converse}

In this subsection we investigate whether or not \emph{any} Lyapunov function of the form (\ref{eq:embededLyapunovFunction}) can be induced from a path-complete graph with as many nodes as the number of pieces of the function itself.
We give a negative answer to this question by providing a counter example from \cite[Example 11]{GoHuDMII}.
Consider the discrete-time linear switching system  on two modes   $x(t+1) = A_{\sigma(t)} x(t)$ with
\begin{equation}
A_1 = \begin{pmatrix}
  0.3 & 1 & 0 \\
  0 & 0.6 & 1\\
  0 & 0 & 0.7
\end{pmatrix}, \,
A_2 = \begin{pmatrix}
0.3 & 0 & 0 \\
-0.5 & 0.7 & 0 \\
-0.2 & -0.5 & 0.7
\end{pmatrix}.
\label{eq:matricesForExample}
\end{equation}

The system has a max-of-quadratics Lyapunov function $V(x)=\max\{V_1(x),V_2(x) \}$,
with $V_i(x) = \left ( x^\top   Q_i x \right )$, $Q_i$ being positive definite matrices. An explicit Lyapunov function is given by\footnote{Such a function can be found numerically by solving the inequalities of \cite[Section 5]{GoHuDMII} for a choice of $\lambda_{:,:,1} = \left ( \begin{smallmatrix} .627 & 0  \\ 1 & 1  \end{smallmatrix} \right )$.}
$$
\begin{aligned}
& Q_1 =
\begin{pmatrix}
  36.95 & -36.91 & -5.58 \\
  \cdot & 84.11 & -38.47\\
  \cdot & \cdot & 49.32
\end{pmatrix}, \\
& Q_2 = \begin{pmatrix}
  13.80 & -6.69 & 4.80 \\
  \cdot & 21.87 & 10.11\\
  \cdot & \cdot & 82.74
\end{pmatrix}.
\end{aligned}
$$

We first observe that these quadratic functions cannot be the solution of a path-complete stability criterion for our example.  Indeed, let us draw the graph of all the valid Lyapunov inequalities. More precisely, we define the graph $\graph = (\{1,2\}, E)$ with two nodes and
\begin{equation}
(i,j,\sigma) \in E \Leftrightarrow A_\sigma^\top Q_j A_\sigma - Q_i \preceq 0,
\label{eq:LMI}
\end{equation}
i.e. the matrix $A_\sigma^\top Q_j A_\sigma - Q_i$ is negative semi-definite.
The graph obtained is presented on Figure \ref{fig:teel}. This graph is not path-complete, and thus we cannot form a Common Lyapunov Function, as done in the previous section, with these two particular pieces.
\begin{figure}[!ht]
\centering
\begin{tikzpicture}[->,>=stealth',shorten >=1pt,auto,node distance=2cm,
                    semithick, scale = 1, transform shape ]
 \node[state] (v1)                          {$V_1$};
 \node[state] (v2)     [right of = v1]      {$V_2$};
  \path (v1) edge [loop left] node {$2$} (v1)
			 edge  node {$2$} (v2)
  		(v2) edge [loop right] node {$1$} (v2)
  		;
\end{tikzpicture}
\caption{The valid Lyapunov inequalities for the quadratic functions for the system and the Lyapunov function in  \cite[Example 11]{GoHuDMII} are represented by the graph above. The graph is not path complete.}
\label{fig:teel}
\end{figure}
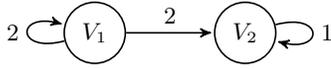

However, we can go further and investigate whether \emph{another} pair of quadratic functions would exist, which we could find by solving a path-complete criterion, and such that their maximum would be a valid CLF.
Recall that co-complete graphs induce Lyapunov functions of the form $\max_{v \in S} V_{v}(x)$ (see Corollary~\ref{cor:CLFcocomplete}).

\begin{proposition}
Consider the discrete-time linear system with two modes (\ref{eq:matricesForExample}).
The system does not have a Path-Complete Lyapunov function with quadratic pieces defined on co-complete graphs with 2 nodes.
\label{prop:noConverse}
\end{proposition}
\begin{proof}
%\comrj{It would be simpler to simply say that we checked the 16 graphs, and none of them works.   But ok maybe some parts are interesting, yet?}
From Definition \ref{def:(co)complete}, there is a total of $16$ graphs that are co-complete and  consist of two nodes and four edges (1 edge per mode and per state).
 We do not examine co-complete graphs with more than four edges since satisfaction of the Lyapunov conditions for these graphs would imply  that of the conditions for at least one graph with four edges.\\
For each graph, the existence of a feasible set of quadratic functions can be tested by solving the LMIs (\ref{eq:LMI}).\\
For the system under consideration, none of the 16 sets of LMIs have a solution. Thus, no induced Lyapunov function of the type $\max\{V_1(x),V_2(x)\}$ exists.
\end{proof}
\begin{remark}
In fact, for the Proof of Proposition \ref{prop:noConverse}, we need only to test four graphs. Three are co-complete with two nodes:
\begin{align*}
G_1 & = (\{a,b\},  \{ (a,a,1), (a,b,1), (b,a,2), (b,b,2) \}), \\
G_{2} & =  (\{a,b\}, \{ (a,a,1), (a,b,1), (a,b,2), (b,a,2) \}), \\
G_{3} & =  (\{a,b\},\{  (a,a,2), (a,b,1), (a,b,2), (b,a,1) \}),
\end{align*}
and the last one corresponds to the common quadratic Lyapunov function
\begin{align*}
G_4 & =( \{a\},  \{ (a,a,1), (a,a,2)  \}).
\end{align*}
One can show that each one of the $13$ remaining co-complete graph is equivalent to one of these four graphs (either isomorphic, or satisfying the conditions of Corollary \ref{cor:bijection} which will be presented later).
\end{remark}

\begin{remark}
For linear systems and for the assessment of asymptotic stability, Path-Complete Lyapunov functions have been shown to be \emph{universal}.
In particular, \cite{LeDuUSOD} show this for the so-called Path-Dependent Lyapunov functions, which are Path-Complete Lyapunov functions with a particular choice of complete graphs, specifically, the so-called \emph{De Bruijn} graphs.\\
The system concerned by Proposition \ref{prop:noConverse} is actually asymptotically stable (see \cite{GoHuDMII}). The interest of Proposition \ref{prop:noConverse} lies in the fact that there do not exist necessarily Path-Complete Lyapunov functions with \emph{the same number of pieces } as a $\max$-type common Lyapunov function.
This is a limitation induced from the combinatorial structure of the Path-Complete Lyapunov function.
\end{remark}

The proof of Proposition \ref{prop:noConverse} highlights an interesting fact. Several different path-complete graphs may induce the same common Lyapunov function (\ref{eq:embededLyapunovFunction}). However,
the strength of the stability certificate they provide may differ. This has a practical implication: if we are given a system of the form (\ref{eq:swsys}), it is unclear which graph $\graph$ we should use to form a Path-Complete Lyapunov function for some number of pieces satisfying a given template (e.g., quadratic functions). We present, in Section \ref{section:partial_order}, a first attempt for analyzing the relative strength of Path-Complete Lyapunov functions based on their graphs and the algebraic properties of the set of functions defining their pieces.

\section{The Partial order on\\ Path-Complete graphs}
\label{section:partial_order}

In this section, we provide tools for establishing an \emph{ordering} between Lyapunov functions defined on general path-complete graphs, extending the work of \cite[Section 4.2]{AhJuJSRA} on \emph{complete} graphs. In the following definition, we introduce $\mathcal{U}$ as a \emph{template} or \emph{family} of functions to which the pieces of Path-Complete Lyapunov functions belong. For example, $\mathcal{U}$ could be the set of quadratic functions: $\mathcal{U} = \{x \mapsto x^\top Q x, \, Q \succ 0\}$. We assume that (\ref{eq:VIsKappa}) holds for any finite subset of $\mathcal{U}$.

\begin{definition} \label{definition_order} (Ordering).
For two path-complete graphs $\graph_1 = (S_1, E_1)$, $\graph_2 = (S_2, E_2)$ and a template $\mathcal{U}$, we write $\graph_1 \leq_\mathcal{U} \graph_2$ if the existence of a Path-Complete Lyapunov function on the graph $\graph_1$ with pieces $(V_s)_{s \in S_1}$, $V_s \in \mathcal{U}$ implies that of a Path-Complete Lyapunov function on the graph $\graph_2$ with pieces $(W_s)_{s \in S_2}$, $W_s \in \mathcal{U}$.
\end{definition}

For each family of functions $\mathcal{U}$, this defines a partial order on path-complete graphs. A minimal element of the ordering, independent of the choice of $\mathcal{U}$, is given by (see Figure \ref{fig:pc*} for $M = 2$)
\begin{equation}
\graph^\star  = (\{a\}, \{(a,a,\sigma)_{\sigma \in [M]}\}).
\label{eq:Gstar}
\end{equation}
A Path-Complete Lyapunov function on this graph  corresponds to the existence of a common Lyapunov function from $\mathcal{U}$ for the system. Thus, $\graph^* \leq_\mathcal{U} \graph$ for any $\mathcal{U}$.

\begin{remark}
We highlight that the properties of the set $\mathcal{U}$ influence the ordering relation defined in Definition \ref{definition_order}. For example, if $\mathcal{U}$ is a singleton, then it is not difficult to see that $\graph_1 \leq_\mathcal{U} \graph_2$ for any two path-complete graphs. From Theorem \ref{thm:EmbeddedCLF}, one can show that this holds  as well for a set $\mathcal{U}$ closed under  $\min$ and $\max $ operations.
\end{remark}
\subsection{Bijections between sets of states}
\label{subsec:Bij}
We present a sufficient condition under which a graph $\graph$ satisfies $\graph \leq_\mathcal{U} \graph^*$.   It is similar in nature to those of Subsection \ref{subsection:HiddenInequalitiesAndProofs}, and requires as well that the set $\mathcal{U}$ is \emph{closed under addition}, an algebraic property satisfied, e.g., by the set of quadratic functions.

\begin{proposition}[Bijection]
Consider a graph $\graph = (S,E)$ feasible for a set of functions $(V_s)_{s \in S}$.  Take two subsets $P$ and $Q$ of $S$. If for $\sigma \in [M]$, there is a subset $E'$ of $E$ such that,
$$
\begin{aligned}
\forall p \in P, \exists! \,  q \in Q: (p,q,\sigma) \in E',\\
 \forall q \in Q, \exists! \, p \in P: (p,q,\sigma) \in E',
 \end{aligned}$$
then
$$ \sum_{q \in Q} V_q(f_\sigma (x)) \leq \sum_{p \in P} V_p(x), \, \forall x \in \reels^n. $$
\label{prop:bijection}
\end{proposition}
\begin{proof}
The result is obtained by first enumerating the $\card{P} = \card{Q}$ Lyapunov inequalities encoded in $E'$, and then summing them up.
\end{proof}

\begin{example}
Consider the graphs $\graph_1 = (S_1, E_1)$ and $\graph_2 = (S_2,E_2)$ of Figure \ref{fig:G*2_2a} and \ref{fig:G*2_2b} respectively.

\begin{figure}[!ht]
\centering
\begin{subfigure}[b]{\columnwidth}
\centering
\begin{tikzpicture}[->,>=stealth',shorten >=1pt,auto,node distance=2.5cm,
                    semithick, scale = 1, transform shape ]
 \node[state] (a)                          {$a$};
 \node[state] (b)     [left of = a]       {$b$};
 \node[state] (c)     [right of = a]       {$c$};

  \path (a) edge [bend right] node [above] {$1$} (b)
  			edge node[above]{$2$} (b)
  			 edge [bend left] node {$1$} (c)
  			edge node{$2$} (c)
  		(b) edge [bend right] node [below] {$1$} (a)
  		(c) edge [bend left] node {$2$} (a)
  	    	;
\end{tikzpicture}
\caption{}
\label{fig:G*2_2a}
\end{subfigure}
\begin{subfigure}[b]{\columnwidth}
\centering
\begin{tikzpicture}[->,>=stealth',shorten >=1pt,auto,node distance=2cm,
                    semithick, scale = 1, transform shape ]
 \node[state] (V1)                           {$a'$};
 \node[state] (V2)     [right of = V1]       {$b'$};
  \path (V1)edge [looseness = 10, loop left]             node {$1$} (V1)
      	    edge [bend left] node {$1$}            (V2)
     	(V2)edge [bend left]             node {$2$} (V1)
      	    edge [loop right] node {$2$} (V2);
\end{tikzpicture}
\caption{}
\label{fig:G*2_2b}
\end{subfigure}
\caption{Example 4.1, the graph $\graph_1$ (8a) and the graph $\graph_2$ (8b). }
\label{fig:G*2_2}
\end{figure}
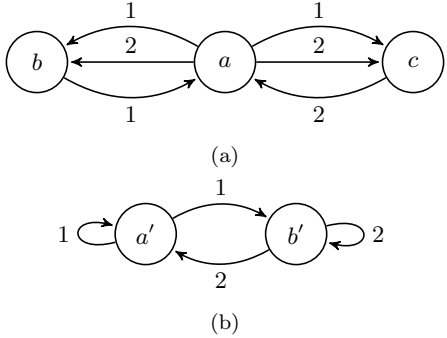
Observe that in $\graph_1$, if we take the two subsets of nodes $R_1 = \{a,b\}$ and $R_2 = \{a,c\}$, then we have that Proposition \ref{prop:bijection} holds for $P = Q = R_1$ and $\sigma = 1$; $P = Q = R_2$ and $\sigma = 2$; $P = R_1,  Q = R_2$ and $\sigma = 1$; and $P = R_2, Q = R_1$ and $\sigma = 2$.\\
Putting together these new Lyapunov inequalities, this allows us to conclude that if  $\{V_a, V_b, V_c\}$ is a solution for $\graph_1$, then $W_{a'} = V_a+V_b$ and $W_{b'} = V_a+V_c$ is a solution for $\graph_2$. Thus, if $\mathcal{U}$ is closed under addition, then it follows that $\graph_1 \leq_\mathcal{U} \graph_2$.
\end{example}

\begin{corollary}
For a graph $\graph = (S,E)$, if for all $\sigma \in [M]$, there exists a subset $E_\sigma \subset E$ such that
$$
\begin{aligned}
\forall p \in S, \exists ! \, q \in S: (p,q,\sigma) \in E_\sigma,\\
 \forall q \in S, \exists ! \, p \in S: (p,q,\sigma) \in E_\sigma,
 \end{aligned}$$
 then if $\graph$ is feasible for $(V_s)_{s \in S}$, the sum
 $\sum_{s \in S} V_s$ is a common Lyapunov function for the system.
 \label{cor:bijection}
\end{corollary}
\begin{proof}
Proposition \ref{prop:bijection} holds for $P = Q = S$ and all $\sigma \in [M]$.
\end{proof}

\begin{example}
Consider the graph $\graph$ of Figure \ref{fig:graph:biject4} on four nodes and two modes.  If $\graph$ is feasible for a set $\{V_a, V_b, V_c, V_d\}$, then the system has a common Lyapunov function given by $V_a+V_b+V_c+V_d$.
\begin{figure}[!ht]
\centering
\begin{tikzpicture}[->,>=stealth',shorten >=1pt,auto,node distance=2cm,
                    semithick, scale = 1, transform shape ]
 \node[state] (a)                         {$a$};
 \node[state] (b)     [right of = a]       {$b$};
 \node[state] (c)     [right of = b]      {$c$};
 \node[state] (d)     [below of = b]      {$d$};
 
  \path (a) edge node {$1$} (b)
  			edge [bend right] node{$2$} (d)
  		(b) edge  [bend left] node [below]{$1$} (c)
  			edge [bend right] node {$2$} (c)
  		(c) edge node {$1$} (d)
  			edge [bend right = 45]node [above]{$2$} (a)
  		(d) edge  node [above] {$1$} (a)
  			edge node [right] {$2$} (b)
  		;
\end{tikzpicture}
\caption{A graph $\graph$ whose feasibility implies that of $\graph^\star$ (Example \ref{ex:bijectionex}).}
\label{fig:graph:biject4}
\end{figure}
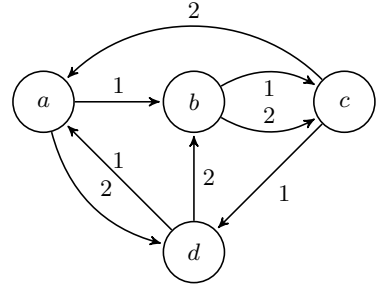
Taking $\mathcal{U}$ as the set of quadratic functions for example, we have $\graph \leq_\mathcal{U} \graph^\star$.
\label{ex:bijectionex}
 \end{example}

\subsection{Ordering by simulation}
\label{subsec:Simu}
This next criterion for ordering is actually independent of the choice of $\mathcal{U}$. It is inspired by the concept of \emph{simulation} between two automata \citep[pp. 91--92]{CaLaITDE}.
\begin{definition}(Simulation)
 Consider two path-complete graphs $\graph_1= (S_1,E_1)$ and $\graph_2= (S_2,E_2)$ with a same  labels $[M]$. We say that \emph{$\graph_1$ simulates $\graph_2$} if there exists a function $F(\cdot):S_2\rightarrow S_1$ such that for any edge $(s_2,d_2,\sigma)\in E_2$ there exists an edge $(s_1,d_1,\sigma)\in E_1$ with $F(s_2)=s_1$, $F(d_2)=d_1$.
 \label{def:simulation}
\end{definition}
\begin{remark}
The notion of simulation we use here is actually stronger than the classical one defined for automata, which defines a \emph{relation} between the states of the two automata rather than a \emph{function}.
\end{remark}
\begin{proposition}\label{prop:simulation}
Consider two graphs $\graph_1 = (S_1,E_1)$ and $\graph_2 = (S_2, E_2)$. If  $\graph_1$ is feasible for  $(V_s)_{s \in S_1}$, and $\graph_1$ simulates $\graph_2$  through the function $F: S_2 \rightarrow S_1$, then  $\graph_2$ is feasible for
$(W_s)_{s\in S_2}$, with
\begin{equation*}
W_s=V_{F(s)},\quad \forall s\in S_2.
\end{equation*}
\end{proposition}
\begin{proof}
Taking any edge $(s,d, \sigma) \in E_2$, we get
$$ W_d(f_\sigma (x)) = V_{F(d)}(f_{\sigma} (x)) \leq V_{F(s)}(x) = W_s(x).$$
\end{proof}
\begin{example}
Consider the graphs   on three modes $\graph_1 = (S_1, E_1)$ and  $\graph_2 = (S_2, E_2)$  on three modes with the first depicted on Fig. \ref{fig:graphSimulating} and  the second on Fig. \ref{fig:graphSimulated}.
Proposition \ref{prop:simulation} applies here with $F(\cdot): S_2 \rightarrow S_1$ defined as $F(a') = a$, $F(b'_1) = F(b'_2) = b$, $F(c') = c$.
\begin{figure}[!ht]
\centering
\begin{subfigure}[b]{\columnwidth}
\centering
\begin{tikzpicture}[->,>=stealth',shorten >=1pt,auto,node distance=2.5cm,
                    semithick, scale = 1, transform shape ]
 \node[state] (a)                         {$a$};
 \node[state] (b)     [right of = a]       {$b$};
 \node[state] (c)     [right of = b]      {$c$};

  \path (a) edge [loop above] node {$2$} (a)
  			edge [loop left] node {$1$} (a)
  			edge [bend left] node {$1$} (b)
  			
  		(b) edge [loop above] node {$2$} (b)
  			edge [bend left] node {$3$} (c)
  			edge [bend left] node {$3$} (a)
  		
  		(c) edge [loop right] node {$3$} (c)
  			edge [bend left] node {$2$} (b)
  			edge [bend left= 45] node {$3$} (a)
  		
  		;
\end{tikzpicture}
\caption{Graph $\graph_1$ for Example \ref{ex:simulation}.}
\label{fig:graphSimulating}
\end{subfigure}
\begin{subfigure}[b]{\columnwidth}
\centering
\begin{tikzpicture}[->,>=stealth',shorten >=1pt,auto,node distance=2.5cm,
                    semithick, scale = 1, transform shape ]
 \node[state] (A)                         {$a'$};

  \node[state] (B)     [right  of = A] {$b'_1$};
 \node[state] (D)	  [right of = B] {$c'$};
  \node[state] (C)     [node distance = 2cm, above  of = B] {$b'_2$};
  \path (A) edge [loop above] node {$2$} (A)
  			edge [loop left] node {$1$} (A)
  			edge [bend left = 15] node {$1$} (B)
  			edge  node {$1$} (C)
  		(B) edge [loop below] node {$2$} (B)
  			edge node{$2$} (C)
  			edge [bend left] node {$3$} (A)  		
  		(C) edge [bend left] node {$3$} (D)
  		(D) edge [loop right] node {$3$} (D)
  			edge [bend left] node {$2$} (B)
  			edge [bend left] node {$2$} (C)
  			edge [bend left= 60] node {$3$} (A)
  		
  		;
\end{tikzpicture}
\caption{Graph $\graph_2$ for Example \ref{ex:simulation}.}
\label{fig:graphSimulated}
\end{subfigure}
\caption{$\graph_1$ simulates $\graph_2$. }
\label{fig:graphsSimulation}
\end{figure}
\label{ex:simulation}
\end{example}
%This observation motivates us to propose, if possible, a structured way of producing ordered graphs which are consistently better in the sense of Definition~\ref{definition_order}, exploiting simulation between graphs. As it turns out, several of such techniques can already be found in the literature (e.g. the path-dependent Lyapunov function hierarchy, see \cite{LeDuUSOD}).\\
%A very general approach, to be detailed in Appendix \ref{appendix:BuildingGraphs}, is as follows: given a graph $\graph_1(S_1, E_1)$, we produce a graph $\graph_2(S_2, E_2)$ simulated by $\graph_1$ through a function $F: S_2 \rightarrow S_1$ such that if $\graph_2$ is $\mathcal{W}$-feasible, then $\graph_1$ is $\mathcal{V}$-feasible either for
%$$V_s(x) = \min \{ (W_s'(x)): \, F(s') = s \}, \forall s \in S_1,$$
%or
%$$V_s(x) = \max \{ (W_s'(x)): \, F(s') = s \}, \forall s \in S_2.$$
% \begin{example}
%One can verify that if the graph $\graph_2$ of Figure \ref{fig:graphSimulated} is $\mathcal{W}$ feasible, then the graph $\graph_1$ of Figure \ref{fig:graphSimulating} is $\mathcal{V}$-feasible for
%$$
%\begin{aligned}
%V_a & = W_A,\\
%V_b & = \max \{W_B, W_C\},\\
%V_c & = W_D.
%\end{aligned}
%$$
%\end{example}

\section{Example and experiment}
\label{Sec:NumericalExample}
In this section, we provide an illustration of our results. First, we present a practically motivated example, where we extract a common Lyapunov function from a Path-Complete Lyapunov function for a given discrete-time linear switching systems on three modes. 
We next present a numerical experiment comparing the performance of three particular path-complete graphs on a testbench of randomly generated systems, similar to that presented in \cite[Section 4]{AhJuJSRA}. 

Our focus is on \emph{linear switching systems}, and Path-Complete Lyapunov functions with \emph{quadratic pieces}. The existence of such Lyapunov functions can then be checked by solving the LMIs (\ref{eq:LMI})\footnote{A Matlab implementation for this example can be found at \url{http://sites.uclouvain.be/scsse/HSCC17_PCLF-AND-CLF.zip}.}.

\subsection{Extracting a Common Lyapunov function.}
The scenario considered here is similar to that of \cite[Section 4]{PhEsSODT}, and deals with the stability analysis of closed-loop linear time-invariant systems subject to failures in a communication channel of a \emph{networked control system} (see e.g. \cite{JuHeCOLS} for more on the topic). We are given a linear-time invariant system of the form $x(t+1)  =  (A + BK) x(t)$, with 
$$ A = \begin{pmatrix} 0.97 & 0.58 \\ 0.17 & 0.5 \end{pmatrix}, \, B = \begin{pmatrix} 0 \\ 1 \end{pmatrix}, K = \begin{pmatrix}
-0.55, 0.24
 \end{pmatrix}.$$

When a communication failure occurs, no signal arrives at the plant, and the control input is automatically set to zero.  In this case, the communication channel needs to be fixed before any feedback signal can reach the plant. In order to prevent the impact of failures, periodic inspections of the channel are foreseen every M steps. However, the inspection of the communication channel is costly, and we would like to compute the largest $M$ such that an inspection of the plant at every $M$ steps is sufficient to ensure its stability.\\
Given $M \geq 1$, we model the failing plant as a switching system with $M$ modes:  
\begin{equation}
x(t+M) = \tilde{A}_{\sigma(t)} x(t),
\label{eq:systemExample}
\end{equation} 
where  $\sigma \in \{1, \ldots, M\}$ and $\tilde{A}_\sigma = A^{\sigma-1}(A+BK)^{M-\sigma+1}$. In other words, for $\sigma(t) = k$, the communication channel will function properly from time $t$ up until time $t + (M-k)$ included,  and will then be down from time $t+(M-k)+1$ until time $t+M-1$ included. This assumes that the channel always functions properly at the very first step after inspection.

For $M = 1$, the stability analysis is direct as $(A+BK)$ is stable. For $M = 2$ we can verify that the system has a \emph{Path-Complete Lyapunov function} for the graph of Figure \ref{fig:G*2_2b}. The case $M = 4$ is straightforward: the matrix $\tilde{A}_4 = A^3(A+BK)$ is unstable, and thus the system is unstable in view of Definition $\ref{def:stability}$.

For the case when $M = 3$, we verify numerically that the system does not have a common quadratic Lyapunov function. Furthermore,it does not have  a Path-Complete Lyapunov function with quadratic pieces for $\graph_2$ on four nodes represented at Figure \ref{fig:graphSimulated}. Note that since  $\graph_1$ simulates $\graph_2$ (see Example \ref{ex:simulation}), this allows us to conclude that $\graph_1$ will not provide us with a Path-Complete Lyapunov function as well, or that would contradict Proposition \ref{prop:simulation}.

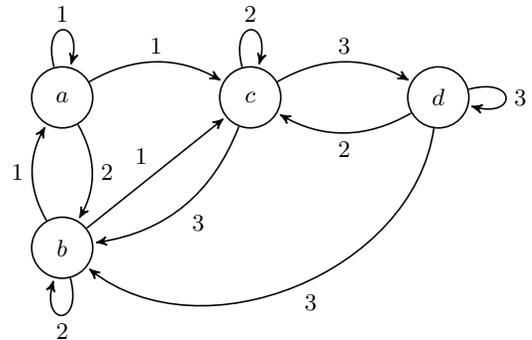
\begin{figure}[!ht]
\centering
\begin{tikzpicture}[->,>=stealth',shorten >=1pt,auto,node distance=2.5cm,
                    semithick, scale = 1, transform shape ]
 \node[state] (a)                         {$a$};
 \node[state] (b)    [node distance = 2cm, below of = a]                     {$b$};
 \node[state] (c)      [right of = a]                   {$c$};
 \node[state] (d)      [right of = c]                   {$d$};
 
  \path (a) edge [loop above] node {$1$} (a)
  			edge [bend left] node {$2$} (b)
  			edge [bend left] node {$1$} (c)
  		(b) edge [loop below] node {$2$} (b)
  			edge [bend left] node{$1$} (a)
  			edge node {$1$} (c)  		
  		(c) edge [loop above] node {$2$} (c)
  			edge [bend left] node {$3$} (b)
  			edge [bend left] node {$3$} (d)
  		(d) edge [loop right] node {$3$} (d)
  			edge [bend left] node {$2$} (c)
  			edge [bend left= 60] node {$3$} (b)
  		
  		;
\end{tikzpicture}
\caption{Graph $\graph_3$, for the example of Section \ref{Sec:NumericalExample}.}
\label{fig:graphUsedForExample}
\end{figure}
However, the graph $\graph_3$ on four nodes represented at Figure \ref{fig:graphUsedForExample}, which alike $\graph_2$ is simulated by $\graph_1$, does provide us with a Path-Complete Lyapunov function with four quadratic pieces.  By applying Theorem \ref{thm:EmbeddedCLF}, after computing the observer graph of $\graph_3$ (see Figure \ref{fig:exampleObserver}),  we obtain a common Lyapunov function $V(x)$ of the form (\ref{eq:embededLyapunovFunction}) for the system,
$$ V(x) = \min_{ S \in \{ \{a,c\}, \, \{b,c\}, \, \{b,d\}\}} \left ( \max_{s  \in S} V_s(x) \right ).$$
whose level set is represented in Figure \ref{fig:exampleLevelSets}.
\begin{figure}[!ht]
\centering
\begin{tikzpicture}[->,>=stealth',shorten >=1pt,auto,node distance=2.5cm,
                    semithick, scale = 1, transform shape ]
 \node[state] (abcd)                         {$a,b,c,d$};

 \node[state, very thick] (bc)      [right of = abcd]                   {$b,c$};
  \node[state, very thick] (ac)    [node distance = 2cm, above of = bc]                     {$a,c$};
 \node[state, very thick] (bd)      [right of = bc]                   {$b,d$};
 
 \path (abcd) edge node {$1$} (ac)
 			  edge [bend right = 50] node [below] {$3$} (bd)
 			  edge  node {$2$} (bc)
       (bc)   edge [bend left] node {$1$} (ac)
       		  edge   node {$3$} (bd)
       		  edge [loop below] node [below] {$2$} (bc)
        (ac)  edge [loop above] node {$1$} (ac)
        	  edge [bend left] node {$3$} (bd)
        	  edge  node {$2$} (bc)
         (bd) edge node {$1$} (ac)
         	  edge  [loop right] node {$3$} (bd)
         	  edge [bend left] node {$2$} (bc)
        ;
\end{tikzpicture}
\caption{Observer graph for $\graph_3$ at Figure \ref{fig:graphUsedForExample}. The subgraph on the nodes $\{a,c\}, \,\{b,c\}$ and $\{b,d\}$ is complete.}
\label{fig:exampleObserver}
\end{figure}
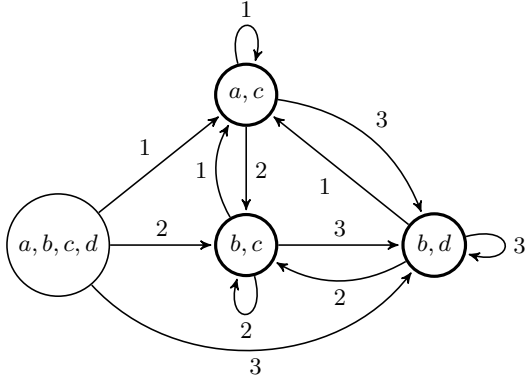
\begin{figure}[!ht]
\centering
\includegraphics[scale = 0.5]{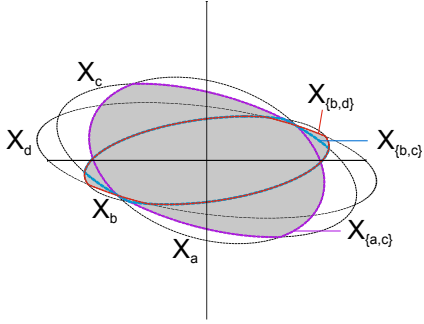}
\caption{The sets $X_a$, $X_b$, $X_c$ and $X_d$ are the level sets of the pieces of the Path-Complete Lyapunov function for the system (\ref{eq:systemExample}). We let $X_{\{i,j\}} = X_i \cap X_j$. The grey set is a level set of a Common Lyapunov function for the system, obtained from the Path-Complete Lyapunov function through Theorem \ref{thm:EmbeddedCLF}.}
\label{fig:exampleLevelSets}
\end{figure}

\label{subsection:CLFExample}

\subsection{Numerical experiment.}

In Section \ref{subsection:CLFExample} we presented a linear switching system for which a Path-Complete Lyapunov function with quadratic pieces exists for the graph $\graph_3$ of Figure \ref{fig:graphUsedForExample}, but not for $\graph_1$ of Figure \ref{fig:graphSimulating} and $\graph_2$ of Figure \ref{fig:graphSimulated}.
 However for another system with three modes, it could be $\graph_2$ that provides a stability certificate, and not $\graph_3$. It is therefore natural to ask which situation is the more likely to occur for random systems.

To this purpose\footnote{For a similar study with other graphs, see \cite[Section 4]{AhJuJSRA}.} we generate triplets of random matrices $M = \{A_1, A_2, A_3\}$\footnote{Each entry of each matrix is the sum of a Gaussian random variable with zero mean and unit variance,  and of a uniformly distributed random variable on $[-1,1]$}. Then, for each triplet and for each graph $\graph_i = (S_i,E_i),\, i = 1,2,3$,  we compute\footnote{This is a quasi-convex optimization program, solved using the numerical tools established in e.g. \cite{AhJuJSRA}.} the quantity
$$\gamma_i = \sup_{\gamma,\, (Q_{s} \succ 0)_{s \in S_i}} \gamma :
\left \{
\begin{aligned}
& \forall (s,d,\sigma) \in E_i,   \\
& \gamma^2 A_\sigma^\top Q_d A_\sigma - Q_s \preceq 0,
 \end{aligned}
 \right . , 
 $$
 that is, the higher number $\gamma$ such that $\graph_i$ provides a stability certificate for the system $x(t+1) = \gamma A_{\sigma(t)} x(t), \, A_\sigma \in M$. 
For a given triplet $M = \{A_1, A_2, A_3\}$, the fact that $\gamma_i \geq \gamma_j$, $i \neq j$, translates as follows:
whenever $\graph_j$ induces a Lyapunov function so does $\graph_i$.
Note that it is  possible that for a triplet $M$, we get $\gamma_i = \gamma_j$.

The results are presented in the Venn diagram of Figure \ref{fig:Venn} for 10800 triplets with matrices of dimension $n = 2$.

\begin{figure}

\centering
\includegraphics[width=0.6\columnwidth]{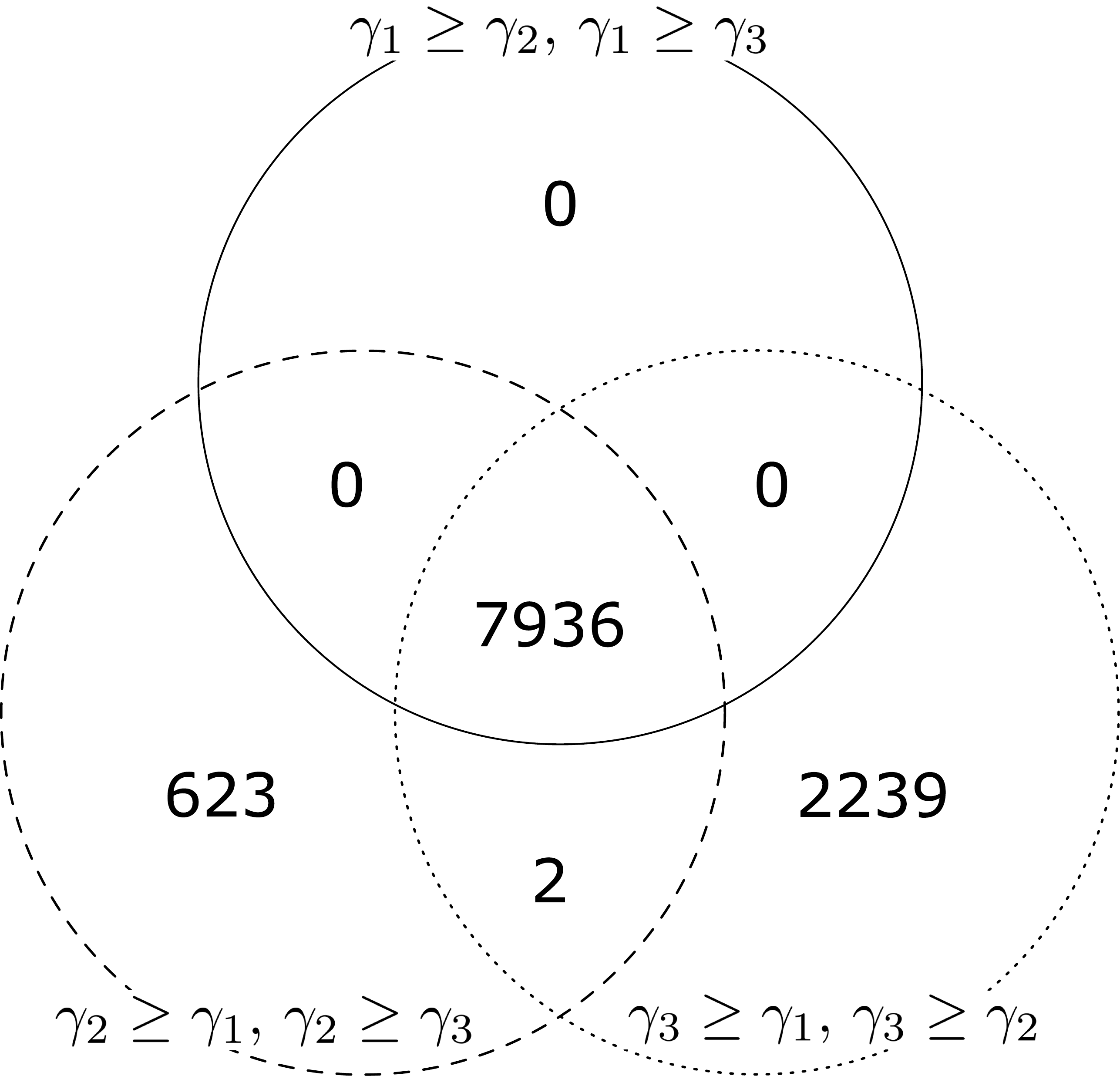}
\caption{Visualization of the outcome of the numerical experiment. As expected from the theory, whenever $\graph_1$ provides a stability certificate, so do  $\graph_2$ and $\graph_3$. There are more systems for which $\graph_3$ provides a certificate and not $\graph_2$ than the reverse. Interestingly, it appears unusual that a system has a stability certificate for both $\graph_2$ and $\graph_3$, and not $\graph_1$.}
\label{fig:Venn}
\end{figure}

We observe that the results are in agreement with Proposition \ref{prop:simulation}, when $\graph_1$ provides a stability certificate, so do $\graph_2$ and $\graph_3$. Also, it appears that a random triplet of matrices is more likely to have a Lyapunov function induced by $\graph_3$ ($\sim 94 \%$ of the cases) rather than by $\graph_2$ ($\sim 79 \%$ of the cases). Interestingly, there appear to be very few instances for which $\gamma_2 = \gamma_3 > \gamma_1$, which deserves further attention.

\section{Conclusion} \label{conclusions}
Path-complete criteria are promising tools for the analysis of hybrid or cyber-physical systems. 
They encapsulate several powerful and popular techniques for the stability analysis of swiching systems.  
However, their range of application seems much wider, as for instance 1) they can handle switching nonlinear systems as well, as it is the case herein, 2) they are not limited to LMIs and quadratic pieces and 3) they have been used to analyze systems where the switching signal is constrained \cite{PhEsSODT}.  On top of this, we are investigating the possibility of studying other problems than stability analysis with these tools.

However, already for the simplest particular case of multiple quadratic Lyapunov functions for switching linear systems, many questions still need to be clarified.  In this paper we first gave a clear interpretation of these criteria in terms of common Lyapunov function: each criterion implies the existence of a common Lyapunov function which can be expressed as the minimum of maxima of sets of  functions.  We then studied the problem of comparing the (worst-case) performance of these criteria, and provided two results that help to partly understand when/why one criterion is better than another one.  We leave open the problem of deciding, given two path-complete graphs, whether one is better than the other.

\bibliographystyle{plain}
\bibliography{biblio}

\end{document}